\documentclass{amsart}
\usepackage{xypic}

\xyoption{all}

\newif\iffurther

\newtheorem{thm}{Theorem}[section] 

\newtheorem{lem}[thm]{Lemma}
\newtheorem{prop}[thm]{Proposition}

\newtheorem{rem}[thm]{Remark}

\newtheorem{ques}[thm]{Question}

\def\[{\left[}
\def\]{\right]}

\def\Ker{{\operatorname{Ker}}}
\def\GKdim{{Gel'fand-Kirillov dimension}}

\long\def\forget#1\forgotten{{}}

\newcommand\isom{{\,\cong\,}}
\def\ra{{\rightarrow}}

\renewcommand\S{{\mathcal{S}}}
\newcommand\Lref[1]{{Lemma~\ref{#1}}}
\def\normali{{\lhd}} 

\begin{document}

\title[]{Affine Algebras with Arbitrary Simple Modules}

\author{Be'eri Greenfeld}
\address{Department of Mathematics, Bar Ilan University, Ramat Gan 5290002, Israel}
\email{beeri.greenfeld@gmail.com}

\thanks{The writer wishes to thank Uzi Vishne and Tomer Bauer.
}

\date{\today}


\begin{abstract}
We construct affine algebras with an arbitrary amount of simple modules of each dimension.
\end{abstract}

\maketitle

\setcounter{tocdepth}{3}

\section{Introduction}

The number of simple modules of any given dimension is a representation-theoretic invariant of an algebra.
Much can be said about the simple modules of algebras satisfying suitable homological finiteness conditions and other restrictions.

Fix an algebraically closed based field $F$. In \cite{Irving}, Irving constructed for any subset $\S \subseteq \{2,3,\dots\}$ a finitely generated $F$-algebra, such that the set of dimensions of simple modules is exactly $\S$, and furthermore it has a unique simple module of each dimension appearing in $\mathcal S$.

Irving's examples are presented in terms of generators and relations. They are thus very concrete, yet do not seem able to be easily modified for arbitrarily many simple modules of each dimension. Taking finite direct sums of Irving's examples, one obtains affine algebras with an arbitrary amount of simple modules of each dimension, provided that these amounts are uniformly bounded for all dimensions.

Our goal in this short note is to introduce a more flexible family of examples, using a different approach. More precisely, given a sequence $(a_n)_{n=2}^{\infty}$ of non-negative integers, we construct a finitely generated $F$-algebra possessing precisely $a_n$ non-isomorphic simple modules of dimension $n$.

\section{Affine algebras with arbitrary simple modules}

Let $F$ be an algebraically closed field. We are forced to make this assumption, because field extensions may occur as simple modules of the algebra. Also, the base field itself plays a role of a simple module, perhaps in infinitely many different ways, so we only deal with modules of dimension greater than one.

\subsection{Construction} \label{construction}

Let $(a_n)_{n=2}^{\infty}$ be a sequence of non-negative integers and $\mathcal{S}$ the support of $(a_n)_{n=2}^{\infty}$, that is, the set of indices $m$ with $a_m\neq 0$. For every $n\in \mathcal{S}$ consider $A_n=\prod_{i=1}^{a_n}M_{n}(F)$, i.e.~$a_n$ copies of $n\times n$ matrix ring over $F$ and let $\rho_i^{(n)}:A_n\rightarrow M_n(F)$ be the projection onto its $i$-th component.

Let $e_n \in A_n$ be the element whose components are matrices with the upper left entry $\lambda_{n,i}$ (where all $\lambda_{n,i} \in F$ are distinct non-zero, and distinct even up to sign for the same $n$) and all other entries zeros.
Let $\sigma_n \in A_n$ be the element whose components are all $$\left( \begin{array}{cc}
0 & I_{n-1} \\
1 & 0 \end{array} \right)$$ where $I_{n-1}$ denotes the identity matrix of order $(n-1)\times (n-1)$.

Finally, let $A=\prod_{n\in \mathcal{S}} A_n$ and let $\pi_n: A \rightarrow A_n$ be the natural projection. Let $e,\sigma \in A$ be the elements whose $n$-th components are $e_n,\sigma_n$ respectively, namely $\pi_n(e) = e_n$ and $\pi_n(\sigma)=\sigma_n$.

Inside $A$, let $B=F\left<e,\sigma\right>$. Note that $F$ is diagonally embedded into $B$, making the latter an $F$-algebra.
For any $n\in \mathcal{S}$ and $1\leq i\leq a_n$ we let $\theta_{n,i}:A \rightarrow M_n(F)$ the composition $\theta_{n,i} = \rho_i^{(n)} \circ \pi_n$, and $\overline{\theta}_{n,i}$ its restriction to $B$.

\subsection{Counting simple modules}

Let $A_0 = \bigoplus_{n\in \mathcal{S}} A_n \subset A$ be the ideal of all elements which are mapped by $\pi_n$ to $0$ for almost all indices $n$. Notice that $A_0$ has $a_n$ simple modules of dimension $n$; however $A_0$ is not affine.

\begin{lem} \label{direct_sum}
We have:
$$A_0 \subset B.$$
\end{lem}

\begin{proof}
First, we claim that it is enough to prove that for all $n\in \mathcal{S}$ and $1\leq i\leq a_n$, $B$ contains an element $x_{n,i}\in A$ such that $\theta_{n,i}(x_{n,i})$ is a matrix whose entries are all zero except for one, and $\theta_{n',i'}(x_{n,i}) = 0$ for any $(n',i') \neq (n,i)$.
Indeed, if such elements exist then $\sum_{i,j=0}^{n} F \sigma^i x_{n,i} \sigma^j $ is the full $i$-th matrix component of $A_n$, and so we obtain that $A_n\subset B$ for all $n$, hence $A_0\subset B$ as claimed.
We now turn to prove the existence of the $x_{n,i}$. This is done by induction.

Let $s_1\in \mathcal{S}$ be the least number in $\mathcal{S}$. Let $e' = e\sigma^{s_1}e$, then $\pi_n(e')=0$ for all $n>s_1$ and $\pi_{s_1}(e') = \pi_{s_1}(e^2)$. Let $1\leq i_0\leq a_{s_1}$ and define $$e'[i_0] = e' \prod_{i \neq i_0} (e' - \lambda_{s_1,i}^2).$$ Then $\theta_{s_1,i}(e'[i_0]) = 0$ for $i\neq i_0$ and $\theta_{s_1,i_0}(e'[i_0])$ has all entries zero except for the upper-left corner which is equal to $$\lambda_{s_1,i_0}^2\prod_{i \neq i_0} (\lambda_{s_1,i_0}^2 - \lambda_{s_1,i}^2)\neq 0.$$

Now for the induction step, assume $n\in \mathcal{S}$ and we already constructed $x_{s,i}$ for all $s\in \mathcal{S}, s<n$ and $1\leq i \leq a_s$.
As in the case of $s_1$ we can define $e' = e\sigma^n e$. We have $\pi_{n'}(e') = 0$ for $n'>n$ and $\pi_n(e') = \pi_n(e^2)$. By induction, there is some $f\in B$ such that $\pi_{n'}(f) = 0$ for all $n \leq n' \in \mathcal{S}$ and $\pi_{n'}(f) = \pi_{n'}(e')$ for all $n > n'\in \mathcal{S}$. Thus $\pi_{n'}(e'' := e' - f) = 0$ for all $n'\neq n$, and $\theta_{n,i}(e'')$ for $1\leq i \leq a_n$ is a matrix with all entries zero except for the upper-left corner. Moreover $\theta_{n,i}\left(e'\prod_{i_0\neq i=1}^{a_n}(e'' - \lambda_{n,i}^2)\right) = 0$ for $i\neq i_0$ and has the desired form (only upper-left entry non-zero) for $i=i_0$.

\end{proof}

Define
$$\overline{B} = F\left<x,y\,|\ xyx = xy^2x = xy^3x = \cdots = 0\right>;$$
there is a projection from $\overline{B}$ to $B/A_0$ defined by $x\mapsto e,\,y\mapsto \sigma$. Indeed, for every $i>0$ we have that $\pi_n(e\sigma^ie) = 0$ for all sufficiently large $n$.
\begin{lem}\label{quotofoverB}
The only finite dimensional prime quotients of $\overline{B}$ have dimension one.
\end{lem}
\begin{proof}
Indeed, let $P \normali \overline{B}$, and consider the quotient $B' = \overline{B}/P$. Since $xyB'xy=0$ by the relations, and $B'$ is prime, $xy \in P$; since $B'$ is generated by the images of $x,y$, it follows that $xB'y = 0$; but again since $B'$ is prime, either $x \in P$ or $y \in P$. In each case, $B'$ is a quotient of the polynomial algebra in one variable, but $F$ is algebraically closed, and since we assume $\dim_F B' < \infty$, so $B' \isom F$.
\end{proof}

\begin{prop} \label{counting}
$B$ has precisely $a_n$ non-isomorphic simple modules of dimension $n$ for all $n\in \mathcal{S}$ and no simple modules of dimension $1< n\notin \mathcal{S}$.
\end{prop}
\begin{proof}
For every $n$, $B$ acts on an $n$-dimensional vector space making it a simple $B$-module through the maps $\theta_{n,i}$. These modules are non-isomorphic to each other, e.g. since $e$ has different eigenvalues acting on each one of them (the distinct $\lambda_{n,i}$).

To prove that there are no other simple modules of $B$ with finite dimensions greater then one, we consider a primitive ideal $P\triangleleft B$ of finite codimension; $P$ is thus maximal. We have to prove that $P = \Ker(\theta_{n,i})$ for some $n\in \mathcal{S}$ and some $1\leq i\leq a_n$.

Recall that by \Lref{direct_sum}, $A_0=\bigoplus_{n\in \mathcal{S}} A_n\triangleleft B$. Now $(A_0+P)/P\triangleleft B/P$ so either $A_0+P=B$ or $A_0\subseteq P$. In the first case, $1\in A_0+P$ so $P$ contains some element with almost all components the identity matrix, so the natural projection $\phi: B\rightarrow B/P$ splits through the map from $B$ to a direct product of some finitely many $A_n$'s. This map is a surjection by \ref{direct_sum} so $\phi$ is in fact defined on a direct product of finitely many of the $A_n$ components, which is finite dimensional and semisimple so $\phi$ coincides with one of the $\theta_{n,i}$.

In the second case, we have the projections $\overline{B} \ra B/A_0 \ra B/P$, so $B/P$ is a prime finite dimensional quotient of $\overline{B}$, which is ruled out by \Lref{quotofoverB}.
\end{proof}

\begin{rem}
We note that $B$ admits no infinite dimensional simple modules. Indeed, let $L\leq B$ be a maximal left ideal with $P\subseteq L\leq B$ the corresponding primitivie ideal (namely, the sum of the two-sided ideals contained in $L$). As in Proposition \ref{counting}, if $A_0\subseteq L$ then $L$ is $1$-codimensional and otherwise $1\in A_0+L$, so $L$ contains an element $x$ with almost all components being the identity matrix, and we may assume the other components are zero, so $x\in L$ is central, hence $x\in P$ and $B/P$ is finite dimensional.

However, if one consider the subalgebra $B[\sigma^{-1}]\subseteq A$ generated by $B$ and $\sigma^{-1}$ then $A_0\triangleleft B[\sigma^{-1}]$ is primitive with quotient isomorphis to $F\left<x,y,y^{-1}|xy^ix=0\ \forall i\neq 0\right>$, so it admits an infinite dimensional simple module.
\end{rem}

\section{Questions}

We conclude with the following questions. Note that the algebras from \ref{construction} are not prime.

\begin{ques}
Let $(a_n)_{n=2}^{\infty}$ be a sequence of non-negative integers. Does there exist an affine prime algbra with precisely $a_n$ non-isomorphic simple modules of dimension $n$?
\end{ques}

In \cite{Agata_Bell}, Bell and Smoktunowicz construct an affine monomial algebra of \GKdim\ two, having simple modules of arbitrarily large finite dimensions. They also prove that no such algebra exists if one requires further quadratic growth.

We find it thus natural to ask:

\begin{ques}
Is there an affine prime algebra with quadratic growth and simple modules of arbitrarily large (finite) dimensions?
\end{ques}

\end{document}